\documentclass{amsart}
\usepackage{tikz}
\usepackage{xcolor}
\usepackage{amssymb,latexsym,amsmath,extarrows,mathabx}
\usepackage{esint}
\usepackage{graphicx,mathrsfs}
\usepackage{hyperref,url}
\usepackage[shortlabels]{enumitem}
\usepackage{makecell}

\usepackage{pgf,tikz,pgfplots}
\usepackage{mathrsfs}
\usetikzlibrary{positioning}
\usetikzlibrary{patterns}
\usepackage{graphicx}

\usepackage{geometry}
\newcommand{\opd}{{\operatorname{d}}}

\newcommand*{\me}{\ensuremath{\mathrm{e}}}
\newcommand*{\mi}{\ensuremath{\mathrm{i}}}
\newcommand*{\supp}{\ensuremath{\mathrm{supp}}}

\newtheorem{theorem}{Theorem}[section]
\newtheorem{lemma}[theorem]{Lemma}

\newtheorem{proposition}[theorem]{Proposition}
\newtheorem{remark}[theorem]{Remark}

\newtheorem{corollary}[theorem]{Corollary}

\newtheorem{conjecture}[theorem]{Conjecture}

\numberwithin{equation}{section}

\pgfplotsset{compat=1.18}

\title{The Szemer\'edi-Trotter Estimate in Finite Field with its Applications}
\author[C. Miao]{Changxing Miao} 
\address{School of Mathematics and Physics, University of Science and Technology Beijing, Beijing 100083, China} 
\email{miao\_changxing@ustb.edu.cn}
\author[R. Xie]{Rui Xie}
\address{The Graduate School of China Academy of Engineering Physics, Beijing 100088, China}
\email{xierui24@gscaep.ac.cn}
\begin{document}
    \begin{abstract}
        We prove a sharp Szemer\'edi-Trotter estimate
        \[\mathcal{I}(A,\mathcal{L})\lesssim \frac{|A||\mathcal{L}|}{p}+|A|^{2/3}|\mathcal{L}|^{2/3}+|A|+|\mathcal{L}|\]
        for prime finite field $\mathbb{F}=\mathbb{F}_p$ by a new polynomial decomposition theorem.
        As applications, we first prove the sharp Furstenberg set estimate in $\mathbb{F}^2$.
        Secondly, we improve sum-product estimate
        \[\max\{|A+A|,|A\cdot A|\}\gtrsim\min\{(p|A|)^{1/2},|A|^{5/4}\},\quad A\subset\mathbb{F}.\]
        Finally, we improve the Fourier restriction estimate $R^*(2\to\alpha)$ holds for $\alpha>\frac{10}{3}$ in $\mathbb{F}^3$ when $p\equiv 3\mod 4$.
    \end{abstract}
    \maketitle

    \section{Introduction}\label{sec:intro}
    \subsection{Overview}
    It was shown by Szemerédi and Trotter \cite{zbMATH03859136}, settling a conjecture by Erd\"os, that points $A\subset \mathbb{R}^2$ and lines $\mathcal{L}$ determine at most $\mathcal{O}(|A|^{2/3}|\mathcal{L}|^{2/3}+|A|+|\mathcal{L}|)$ point-line incidences, which defined by
    \[\mathcal{I}(A,\mathcal{L}):=\sum_{a\in A}|\{\ell\in\mathcal{L}:a\in\ell\}|=\sum_{\ell\in\mathcal{L}}|\{a\in A:a\in \ell\}|.\]
    In the past twenty years, there are a lot progress in incidence geometry coming form the polynomial method since \cite{zbMATH05817590}, such as Kaplan-Matou{\v{s}}ek-Sharir \cite{zbMATH06093492} have used Guth-Katz's polynomial method to prove Szemerédi-Trotter theorem and
    Zahl \cite{zbMATH06494224} have used polynomial method to prove the Szemerédi-Trotter theorem in $\mathbb{C}^2$.
    Recently, Shen \cite{shen2025szemereditrottertheoremarbitraryfield} have prove the Szemerédi-Trotter theorem for all characteristic zero field, e.g. $p$-adic plane $\mathbb{Q}_p^2$.

    For $\mathbb{F}^2$, where $\mathbb{F}=\mathbb{F}_p$, $p>2$ is prime, the same $\mathcal{O}(|A|^{2/3}|\mathcal{L}|^{2/3}+|A|+|\mathcal{L}|)$ bound fails if we choose $A=\mathbb{F}^2$ and $\mathcal{L}$ is the collection of all non-vertical lines.
    In fact, $|A|=|\mathcal{L}|=p^2$ and $\mathcal{I}(A,\mathcal{L})=p^3$.
    It shows a big difference between lines in $\mathbb{R}^2$ and $\mathbb{F}^2$.
    Since the line in $\mathbb{F}^2$ only contains finite $p$ points, when $|A|$ is large and randomly distributed, one may expect each line contains $\frac{|A|}{p}$ points in $A$, and the incidence will be dominated by $\frac{|A||\mathcal{L}|}{p}$.
    
    The incidence estimate in $\mathbb{F}^2$ was first studied by Bourgain-Katz-Tao \cite{zbMATH02121750}.
    They proved that if we add a restriction $|A|,|\mathcal{L}|\leq p^\alpha$, $0<\alpha<2$, then 
    \[\mathcal{I}(A,\mathcal{L})\lesssim p^{\frac{3\alpha}{2}-\varepsilon}\]
    holds for some $\varepsilon=\varepsilon(\alpha)>0$.
    For the special case $|A|=|\mathcal{L}|<p$, Helfgott-Rudnev \cite{zbMATH05851269} proved a lower bound $\varepsilon\geq\frac{1}{10678}$.
    And then improved to $\varepsilon\geq\frac{1}{662}$ by Jones \cite{jones2012improvementsincidencebecktypebounds}.

    After Bourgain-Katz-Tao, Vinh \cite{zbMATH05982464} proved a incidence estimate without restriction:
    For any $A\subset \mathbb{F}^2$ and lines $\mathcal{L}$, we have
    \[\mathcal{I}(A,\mathcal{L})\lesssim\frac{|A||\mathcal{L}|}{p}+(p|A||\mathcal{L}|)^{1/2}.\]
    And then, Lund-Pham-Vinh \cite{lund2026orthogonalprojectionsincidencebounds} put forward a conjecture on this problem by adding the additional term $\frac{|A||\mathcal{L}|}{p}$ in the original Szemerédi-Trotter bound.
    In this paper, we prove their conjecture.
    \begin{theorem}\label{thm:szemereditrotter}
        For any $A\subset\mathbb{F}^2$ and lines $\mathcal{L}$,
        \begin{equation}\label{eq:szemereditrotter}
            \mathcal{I}(A,\mathcal{L})\lesssim \frac{|A||\mathcal{L}|}{p}+|A|^{2/3}|\mathcal{L}|^{2/3}+|A|+|\mathcal{L}|.
        \end{equation}
    \end{theorem}

    To prove Theorem \ref{thm:szemereditrotter}, we prove an equivalent theorem.
    For $A\subset \mathbb{F}^2$ and $2\leq r\leq |A|$, define 
    \[\mathcal{L}_r(A):=\{\ell\subset \mathbb{F}^2:\ell\text{ is a line such that }|A\cap\ell|\geq r\}.\]
    \begin{theorem}\label{thm:rrichF}
        For $A\subset\mathbb{F}^2$.
        If $r\geq \max\left\{2,\frac{16|A|}{p}\right\}$, then we have 
        \begin{equation}\label{eq:rrichF}
            |\mathcal{L}_r(A)|\lesssim \frac{|A|^2}{r^3}+\frac{|A|}{r}.
        \end{equation}
    \end{theorem}
    It's clear that if we choose $\mathcal{L}=\mathcal{L}_r(A)$, then Theorem \ref{thm:szemereditrotter} implies Theorem \ref{thm:rrichF}.
    It suffices to prove that Theorem \ref{thm:rrichF} implies Theorem \ref{thm:szemereditrotter}. 
    See Section \ref{sec:incidence} for more detail.

    Our proof of Theorem \ref{thm:rrichF} is based on the polynomial method developed by Guth-Katz \cite{zbMATH06383662}, where they reduce the estimate of number of rotation in $\mathbb{R}^2$ to an estimate of rich points in $\mathbb{R}^3$.
    Inspired by their reduction, we can lift the non-vertical lines $\ell_{a,b}:=\{(x,ax+b),x\in\mathbb{F}\}$ in $\mathbb{F}^2$ to $\tilde{\ell}_{a,b}:=\{(x,ax+b,a):x\in\mathbb{F}\}$ in $\mathbb{F}^3$.
    And we can find a polynomial $P\in\mathbb{F}[x,y,z]$ with bounded degree with respect to $r$ and $|A|$ such that $\tilde{\ell}_{a,b}\subset Z(P)$ for all non-vertical $\ell_{a,b}\in\mathcal{L}_r(A)$.
    By constructing the key polynomial decomposition adapted to the lifted line, we can find a polynomial $Q\in\bar{\mathbb{F}}[a,b]$ such that the dual point $(a,b)\in\mathbb{F}^2$ lies on $Z(Q)$ for all non-vertical $\ell_{a,b}\in\mathcal{L}_r(A)$, except a few points for which there is a good upper bound on their number by the degrees of $P$.
    Combining it with an estimate determined by B\'ezout theorem for $r$-rich points on $Z(Q)$, we can obtain Theorem \ref{thm:rrichF} for small $r$.
    For the high rich case, we use an alternative C\'ordoba $L^2$ argument.

    The most important difference between polynomial decomposition in this paper and Guth-Katz's polynomial decomposition is that after roling out the degenerate part of $Z(P)$, which exactly is $Z(Q)$, we can use multihomogeneous B\'ezout theorem \cite[Example 4.9]{zbMATH06176082} to count the intersection of varieties in multiprojective space.
    See the discussion in Section \ref{sec:poly} for more detail.

    \subsection{Applications}
    We will give three important applications of Theorem \ref{thm:szemereditrotter} and Theorem \ref{thm:rrichF}.
    We first prove the sharp Furstenberg set estimate in $\mathbb{F}^2$.
    Secondly, we improve sum-product estimate
    \[\max\{|A+A|,|A\cdot A|\}\gtrsim\min\{(p|A|)^{1/2},|A|^{5/4}\},\quad A\subset\mathbb{F}.\]
    Finally, we improve the Fourier restriction estimate $R^*(2\to\alpha)$ holds for $\alpha>\frac{10}{3}$ in $\mathbb{F}^3$ when $p\equiv 3\mod 4$.
    \subsubsection{Furstenberg set estimate}
    The Furstenberg set problem was originally considered by Wolff \cite{zbMATH01303711} in $\mathbb{R}^2$.
    Fix $s\in(0,1].\; t\in(0,2]$.
    We say $E\subset\mathbb{R}^2$ is a $(s,t)$-set in $\mathbb{R}^2$. if $E$ is of the form 
    \[E=\bigcup_{\ell\in\mathcal{L}}Y(\ell),\]
    where $\mathcal{L}$ is a set of lines in $\mathbb{R}^2$ with $\dim_H\mathcal{L}\geq t$ and for each $\ell\in\mathcal{L}$ there exists $Y(\ell)\subset \ell$ with $\dim_H Y(\ell)\geq s$.
    The Furstenberg set problem asks about the optimal lower bound on the Hausdorff dimension of $(s,t)$-set.
    Ren-Wang \cite{ren2025furstenbergsetsestimateplane} resolved the Furstenberg set problem by proving 
    \[\min_{E:E\text{ is }(s,t)\text{-set}}{\dim_H} E=\min\left\{s+t,\frac{3s}{2}+\frac{t}{2},s+1\right\}.\]

    For lines $\mathcal{L}$ in $\mathbb{F}^2$.
    We say $E\subset \mathbb{F}^2$ is $r$-dense Furstenberg set with respect to $\mathcal{L}$ iff. 
    $|Y(\ell)|:=|E\cap\ell|\geq r$ holds for all $\ell\in\mathcal{L}$.
    Gan \cite{zbMATH08226297} proposed corresponding Furstenberg set conjecture in $\mathbb{F}^2$, which is a direct corollary of Theorem \ref{thm:rrichF}.
    \begin{theorem}\label{thm:furstenberg}
        If $E\subset \mathbb{F}^2$ is $r$-dense Furstenberg set with respect to $\mathcal{L}$, then 
        \[|E|\gtrsim r^{3/2}|\mathcal{L}|^{1/2}+r|\mathcal{L}|+pr.\]
    \end{theorem}
    Gan also proved that Theorem \ref{thm:furstenberg} implies the exceptional set estimate in $\mathbb{F}^n$ for $n\geq 2$.

    For integer $1\leq k<n$, let $G(k,\mathbb{F}^n)$ denote the Grassmannian of $k$-dimensional linear subspace of $\mathbb{F}^n$.
    Given $A\subset \mathbb{F}^n$, the orthogonal projection of $A$ along $V\in G(n-k,\mathbb{F}^n)$ is defined to be
    \[\pi_V(A):=\{x+V:(x+V)\cap A\neq \varnothing, \; x\in\mathbb{F}^n\}.\]
    For $s>0$, define the exceptional set 
    \[E_s(A;n,k):=\{V\in G(n-k,\mathbb{F}^n):|\pi_V(A)|\leq p^s\}.\]
    We have the following proposition. 

    \begin{proposition} 
        Fix $1\leq k<n$, $a\in(0,n]$ and $s>0$.
        For $A\subset\mathbb{F}^n$ such that $|A|\geq p^a$, we have 
        \[|E_{s-\varepsilon}(A;n,k)|\lesssim p^{\varepsilon+\mathbf{M}(a,s;n,k)}.\]
        Here $\mathbf{M}(a,s;n,k)$ is explicit.
        Write $a=m+\beta$ and $s=l+\gamma$, with $m,l$ nonnegative integers and $0<\beta,\gamma\leq 1$.
        \[\mathbf{M}(a,s;n,k)=\begin{cases}
            -\infty,\quad s\leq a-(n-k),\\
            k(n-k),\quad s>\min\{a,k\},\\
            k(n-k)-(m-l)(k-l)+\max\{2\gamma-\beta-1,0\},\;\; a-(n-k)<s\leq\min\{a,k\},\,\gamma>\beta,\\
            k(n-k)-(m+1-l)(k-l)+\max\{2\gamma-\beta,0\},\;\; a-(n-k)<s\leq\min\{a,k\},\,\gamma\leq\beta.
        \end{cases}\]
    \end{proposition}
    In particular, when $n=2$ and $k=1$, it impiles the conjecture by Chen \cite{zbMATH06854663}.
    \begin{corollary}
        Fix $a\in(0,2]$ and $s>0$.
        For $A\subset\mathbb{F}^2$ such that $|A|\geq p^a$, we have 
        \[|E_s(A;2,1)|\lesssim p^{\max\{0,2s-a\}}.\] 
    \end{corollary}

    \subsubsection{Sum-product estimates}
    Theorem \ref{thm:szemereditrotter} can be used to improve the sum-product estimate in $\mathbb{F}$.

    Erd\"os-Szemer\'edi \cite{zbMATH03834055} seek to establish that for $A\subset\mathbb{Z}\text{ or }\mathbb{R}$ and any $\varepsilon>0$,
    \[\max\{|A+A|,|A\cdot A|\}\gtrsim_\varepsilon |A|^{2-\varepsilon}.\]
    Here $A+A=\{a+b:a,b\in A\}$ and $A\cdot A:=\{a\cdot b:a,b\in A\}$.
    In $\mathbb{R}$, Elekes \cite{zbMATH01100474} instigated the use of tools from incidence geometry in the study of the sum-product problem.
    Specifically, he proved Szemer\'edi-Trotter theorem in $\mathbb{R}^2$ can imply 
    \[\max\{|A+A|,|A\cdot A|\}\gtrsim |A|^{5/4},\quad {\text{for finite }}A\subset \mathbb{R}.\]
    Naturally, one might ask whether there are similar sum-product estimate on finite field $\mathbb{F}$.
    Bourgain-Katz-Tao \cite{zbMATH02121750} proved that if $p^\alpha<|A|<p^{1-\alpha}$ holds for some $0<\alpha<1/2$, then there esists $\varepsilon=\varepsilon(\alpha)>0$ such that 
    \[\max\{|A+A|,|A\cdot A|\}\gtrsim_\alpha |A|^{1+\varepsilon}.\]
    Notably, for any $1\leq N\leq p$, Garaev \cite{zbMATH05308790} constructed set $A\subset \mathbb{F}$ such that $|A|=N$ and 
    \begin{equation}\label{eq:sumexample}
        \max\{|A+A|,|A\cdot A|\}\lesssim p^{1/2}N^{1/2}.
    \end{equation}
    Garaev also proved the lower bound 
    \[\max\{|A+A|,|A\cdot A|\}\gtrsim \min\{|A|^{2}p^{-1/2},|A|^{1/2}p^{1/2}\},\]
    which, as (\ref{eq:sumexample}) shows, is sharp in the range $|A|>p^{2/3}$.
    When $|A|\lesssim p^{1/2}$, the best known result is due to Mohammadi-Stevens \cite{zbMATH07672857}, they proved 
    \[\max\{|A+A|,|A\cdot A|\}\gtrsim (\log|A|)^{\mathcal{O}(1)} |A|^{5/4}.\]
    By Theorem \ref{thm:szemereditrotter} and Elekes' method, we can improve Mohammadi-Stevens' $5/4$ bound without $\log$-loss to the whole remaining range $|A|\leq p^{2/3}$.
    \begin{theorem}\label{thm:sumproduct}
        For $A\subset \mathbb{F}$, we have 
        \[\max\{|A+A|,|A\cdot A|\}\gtrsim\min\{(p|A|)^{1/2},|A|^{5/4}\}.\]
        In particluar, if $|A|\leq p^{2/3}$, then 
        \[\max\{|A+A|,|A\cdot A|\}\gtrsim|A|^{5/4}.\]
    \end{theorem}

    For finite set $A$, consider the product set of the difference set
    \[(A-A)\cdot(A-A):=\{(a-b)\cdot(c-d):a,b,c,d,\in A\}.\]
    One expects that $(A-A)\cdot(A-A)$ will always be large in comparison to the input set $A$.
    In $\mathbb{R}$, Netwon-Rudnev \cite{zbMATH06509375} proved that 
    \[|(A-A)\cdot(A-A)|\gtrsim \frac{|A|^2}{\log|A|},\quad \text{for finite }A\subset\mathbb{R}.\]
    In $\mathbb{F}$, we have the similar estimate.
    \begin{theorem}\label{thm:diffproduct}
        For $A\subset \mathbb{F}$ satisfies $|A|\geq 2$, there exists a subset $B\subset A^2$ with $|B|\geq |A|^2/2$ such that for any $(a,b)\in B$, 
        \[|(A-a)\cdot(A-b)|\gtrsim \min\left\{\frac{|A|^2}{\log|A|},p\right\}.\]
        In particlur, 
        \[|(A-A)\cdot(A-A)|\gtrsim \min\left\{\frac{|A|^2}{\log|A|},p\right\}.\]
    \end{theorem}

    \subsubsection{Fourier restriction problem}
    The Fourier restriction conjecture in $\mathbb{R}^d$ is proposed by Stein \cite{zbMATH03695779}.
    \begin{conjecture}\label{conj:rrestriction}
        Assume $\mathbb{P}^{d-1}$ is the paraboloid in $\mathbb{R}^d$, $\opd\sigma$ is the surface measure on $\mathbb{P}^{d-1}$.
        For function $f:\mathbb{P}^{d-1}\to\mathbb{C}$, 
        \begin{equation}\label{eq:rrestriction}
            \|(f\opd\sigma)^\vee\|_{L^\alpha(\mathbb{R}^d)}\lesssim\|f\|_{L^\beta(\mathbb{P}^{d-1},\opd\sigma)}
        \end{equation}
        holds iff. $\alpha>\frac{2d}{d-1}$ and $\alpha\geq\frac{d+1}{d-1}\beta'$.
        Here $\beta'$ is the Lebesgue conjugate of $\beta$.
    \end{conjecture}
    Stein-Tomas \cite{zbMATH03467755, zbMATH03957000} have proved that (\ref{eq:rrestriction}) holds for $\alpha\geq \frac{2(d+1)}{d-1},\; \beta=2$.
    And Tao \cite{zbMATH02067279} have used bilinear method to prove 
    \[\||(f_1\opd\sigma)^\vee(f_2\opd\sigma)^\vee|^{1/2}\|_{L^\alpha}\lesssim\|f_1\|_{L^2}^{1/2}\|f_2\|_{L^2}^{1/2},\quad {\rm{dist}}(\supp(f_1),\supp(f_2))\sim 1\]
    holds for $\alpha>\frac{2(d+2)}{d}$.
    By a bilinear reduction, he obtained (\ref{eq:rrestriction}) holds for $\alpha>\frac{2(d+2)}{d}, \beta'=\frac{d-1}{d+1}\alpha$.
    In particular, when $d=3$, it holds for $\alpha>10/3$.
    The best known result for Conjecture \ref{conj:rrestriction} when $d=3$ is due to Wang-Wu \cite{wang2024restrictionestimatesusingdecoupling}, where they proved (\ref{eq:rrestriction}) holds for $\alpha=\beta>22/7$.

    Mockenhaupt-Tao \cite{zbMATH02103577} first studied restriction estimate in $\mathbb{F}^3$.
    Let $R^*(2\to \alpha)$ denote the estimate 
    \[\|(f\opd\sigma)^\vee\|_{L^\alpha(\mathbb{F}^3)}\lesssim \|f\|_{L^2(\mathbb{P}^2,\opd\sigma)}\]
    Here $\opd \sigma$ is the surface measure on paraboloid $\mathbb{P}^2$ in $\mathbb{F}_*^3$ (See the definition in Section \ref{sec:fourier}).
    When $p\equiv 1\mod 4$, $\mathbb{P}^2$ will be isotropic because $-1$ is square.
    Then the estimate $R^*(2\to 4)$ is sharp.
    For $p\equiv 3\mod 4$, Mockenhaupt-Tao put forward the following conjecture.
    \begin{conjecture}[\cite{zbMATH02103577}]\label{conj:restriction}
        If $p\equiv 3\mod 4$, then $R^*(2\to\alpha)$ holds for $\alpha\geq3$.
    \end{conjecture}
    In \cite{zbMATH02103577}, they proved that if $p\equiv 3\mod 4$, then $R^*(2\to 18/5+\varepsilon)$ holds.
    The best known result is due to Lewko \cite{lewko2026bilinearapproachfinitefield}, where he use Mockenhaupt-Tao's bilinear method and some incidence estimates to prove $R^*(2\to 176/51+\varepsilon)$ for $p\equiv 3\mod 4$.
    By our improved incidence estimate and bilinear argument, we can prove the following restriction estimate.
    \begin{theorem}\label{thm:restriction}
        If $p\equiv 3\mod 4$, then $R^*(2\to\alpha)$ holds for $\alpha>\frac{10}{3}$.
    \end{theorem}
    We remark that Lewko \cite{lewko2026bilinearapproachfinitefield} pointed out that Conjecture \ref{conj:restriction} could imply the integers on paraboloid $\Gamma:=\{(n_1,n_2,n_1^2+n_2^2):n_1,n_2\in\mathbb{Z}\}$ and integers on sphere $S_R:=\{n\in\mathbb{Z}^3:n_1^2+n_2^2+n_3^2=R\}$ for $R\neq 0$ are both $\Lambda(3)$-set.
    Bourgain \cite{zbMATH00404289} predicts $\Gamma$ is $\Lambda(r)$-set for all $r<4$ and $S_R$ is $\Lambda(r)$-set for all $r<6$.
    The wider range for the sphere reflecting the greater arithmetic sparsity of the discrete sphere.

    \subsection{Notations} 
    Throghout this paper, we use $|A|$ to denote the number of elements of finite set $A$.
    For $A,B\geq 0$, we use $A\lesssim B$ to denote $A\leq CB$ for absolute (large) constant $C> 0$.
    And we use $A\sim B$ to mean $A\lesssim B$ and $B\lesssim A$.
    For $x\in\mathbb{F}^n$, we write $x=(\bar{x},x_n)\in\mathbb{F}^{n-1}\times \mathbb{F}$.
    For function $P$, we denote the zero set of $P$ as $Z(P)$.

    \section{Polynomial Decomposition}\label{sec:poly}
    Assume $\mathfrak{K}$ is algebraically closed field.
    For polynomial $P\in\mathfrak{K}[x,y,z]$, let $\deg_{x,y}P$ be the total degree of $P$ with respect to variables $x,y$ and $\deg_z P$ be the degree of $P$ with respect to variable $z$.
    Let 
    \[\mathcal{S}(P):=\{(a,b)\in\mathfrak{K}^2:P(x,ax+b,a)\equiv 0 \text{ in } \mathfrak{K}[x]\}.\]
    For line $\ell_{a,b}:=\{(x,ax+b):x\in\mathfrak{K}\}$, define the lifted line $\tilde{\ell}_{a,b}:=\{(x,ax+b,a):x\in\mathfrak{K}\}$.
    Then $(a,b)\in \mathcal{S}(P)$ is equivalent to $\tilde{\ell}_{a,b}\subset Z(P)$.

    We construct a decomposition of $\mathcal{S}(P)$.
    Consider the directional derivative $\mathcal{D}=\partial_x+z\partial_y$ along the lifted line.
    Then $\mathcal{D}P|_{\tilde{\ell}_{a,b}}\equiv 0$ for $(a,b)\in\mathcal{S}(P)$.
    Let $Q$ be the product of the irreducible polynomial factor $F_j$ of $P$ such that $\mathcal{D}F_j\equiv 0$.
    Then $Z(Q)$ is the degenerate part of $Z(P)$.
    Denote $P'=\frac{P}{Q}$, it suffices to consider $P'$.
    It's clear that $\tilde{\ell}_{a,b}$ is the common root of $P'$ and $\mathcal{D}P'$ for $(a,b)\in\mathcal{S}(P')$.
    In other words, 
    \begin{equation*}
        \bigcup_{(a,b)\in\mathcal{S}(P')}\tilde{\ell}_{a,b}\subset Z(P')\cap Z(\mathcal{D}P').
    \end{equation*}
    Note that $\gcd(P',\mathcal{D}P')=1$.
    By B\'ezout theorem, we obtain 
    \begin{equation}\label{eq:bezoutbound}
        |\mathcal{S}(P')|\lesssim (D_1+D_2)^2.
    \end{equation}

    On the other hand, there is an additional structure of the lifted lines $\tilde{\ell}_{a,b}$.
    Consider the standard open immersion:
    \begin{align*}
        \iota: \mathfrak{K}^3&\hookrightarrow\mathbb{P}^2\times\mathbb{P}^1\\
        (x,y,z)&\mapsto ([x:y:1],[z:1])
    \end{align*}
    Here we write the point in $\mathbb{P}^2$ under homogeneous coordinates to be $[X:Y:W]$ and point in $\mathbb{P}^1$ to be $[R:T]$
    When $W\neq 0$ ($T\neq 0$), the corresponding affine coordinates are $x=\frac{X}{W},\; y=\frac{Y}{W}$ ($z=\frac{R}{T}$).
    Hence the lifted line $\tilde{\ell}_{a,b}$ could be rewritten under homogeneous coordinates:
    \[\bar{\ell}_{a,b}:=\{Y-aX-bW=0\}\times\{[a:1]\},\]
    which is a product of lines in $\mathbb{P}^2$ with a fixed point in $\mathbb{P}^1$.
    Denote $d_1:=\deg_{x,y}P'\leq D_1$ and $d_2:=\deg_zP'\leq D_2$.
    Extend $P'$ to infinity by 
    \[P'^h(X,Y,W;R,T):=W^{d_1}T^{d_2}P'\left(\frac{X}{W},\frac{Y}{W},\frac{R}{T}\right).\] 
    And similarly extend $\mathcal{D}P'$ to infinity by $(\mathcal{D}P')^h$.
    It suffices to count the number of $\bar{\ell}_{a,b}\subset Z(P'^h)\cap Z((\mathcal{D}P')^h)$.

    Consider the slice 
    \[\left(\bigcup_{(a,b)\in\mathcal{S}(P')}\bar{\ell}_{a,b}\right)\cap \left(L\times\mathbb{P}^1\right),\]
    where $L:=\{\alpha X+\beta Y+\gamma W=0\}$ is line in $\mathbb{P}^2$ for some $\alpha,\beta,\gamma\in\mathfrak{K}$ such that $\alpha+\beta a\neq 0$ holds for all $(a,b)\in\mathcal{S}(P')$.
    We can always find such $\alpha,\beta$ because there are only finite $a$ by (\ref{eq:bezoutbound}), but $\mathfrak{K}$ is infinite.
    Then for each $(a,b)\in\mathcal{S}(P')$, $\bar{\ell}_{a,b}\cap (L\times\mathbb{P}^1)$ is an unique point.
    In the affine coordinates, we can express this point clearly:
    Solving $\alpha x+\beta y+\gamma=0$ such that $x=t,\; y=at+b,\;z=a$, we obtain $t_{a,b}=-\frac{\beta b+\gamma}{\beta a+\alpha}$.
    This unique point can be expressed as $(t_{a,b},at_{a,b}+b,a)$.

    Now it suffices to count the number of solutions of equations:
    \begin{equation}\label{eq:resultingequations}
        \begin{cases}
            P'^h(X,Y,W;R,T)=0\\
            (\mathcal{D}P')^h(X,Y,W;R,T)=0\\
            \alpha X+\beta Y+\gamma W=0
        \end{cases}.
    \end{equation}
    By multihomogeneous B\'ezout theorem \cite[Example 4.9]{zbMATH06176082} on $\mathbb{P}^2\times\mathbb{P}^1$, the number of solutions of (\ref{eq:resultingequations}) $\leq (d_1-1)d_2+d_1(d_2+1)\lesssim D_1D_2$.
    
    In summary, we obtain 
    \[|\mathcal{S}(P)\setminus Z(Q)|\leq|\mathcal{S}(P')|\lesssim D_1D_2.\]
    This inequality improves the previous bound (\ref{eq:bezoutbound}).

    \medskip

    Now we give an elementary proof of the above argument.
    \begin{theorem}\label{thm:polynomialdec}
        Assume non-zero polynomial $P\in\mathfrak{K}[x,y,z]$ satisfies $\deg_{x,y}P\leq D_1$ and $\deg_z P\leq D_2$, where $\mathfrak{K}$ has characteristic zero or characteristic $p>D_1$ and $D_2\geq 1$.
        Then there exists non-zero polynomial $Q\in\mathfrak{K}[a,b]$ with total degree $\leq D_2$ such that 
        \begin{equation}
            |\mathcal{S}(P)\setminus Z(Q)|\lesssim D_1D_2.
        \end{equation}
    \end{theorem}
    \begin{proof}
        
        \noindent{\bf{Step I.}} Construction of $Q$

        Witout loss of generality, we can assume $P$ is non-constant and square-free.
        Let $P=\prod_{j=1}^m F_j$, where $\{F_j\}$ are coprime irreducible polynomials.
        Consider differential operator $\mathcal{D}:=\partial_x+z\partial_y$.
        Then $F_j|\mathcal{D}F_j$ is equivalent to $\mathcal{D}F_j=0$.
        This is because if $\mathcal{D}F_j\neq 0$, then $\deg_{x,y}(\mathcal{D}F_j)\leq \deg_{x,y}F_j-1$.

        Fix a $F_j$ such that $\mathcal{D}F_j=0$.
        Define $\tilde{F}_j(x,b,z):=F_j(x,zx+b,z)$.
        Then $\partial_x \tilde{F}_j=\mathcal{D}F_j=0$, and $\deg_x\tilde{F}_j=\deg_x F_j(x,zx+b,z)\leq \deg_{x,y}F_j\leq D_1<p$.
        Thus $\tilde{F}_j$ is independent of $x$.
        Let 
        \[Q_j(a,b)=\tilde{F}_j(0,b,a)=\tilde{F}_j(x,b,a)=F_j(x,ax+b,a).\]
        It's clear that $F_j(x,ax+b,a)\equiv 0$ in $\mathfrak{K}[x]$ is equivalent to $Q_j(a,b)=0$.
        Now we prove $\deg Q_j=\deg_zF_j$.
        In fact, if we assume $Q_j(a,b)=\sum_{i+k\leq \deg Q_j}c_{ik}a^ib^k$, then $F_j(x,y,z)=\sum_{i+k\leq \deg Q_j}c_{ik}z^i(y-xz)^k$.
        The coefficient of $z^{\deg Q_j}$ is $\sum_{i+k=\deg Q_j}c_{ik}(-x)^k$, which is non-zero polynomial of $x$.

        Let $Q:=\prod_{j:\mathcal{D}F_j=0}Q_j$, then $\deg Q\leq D_2$.
        
        \medskip

        \noindent{\bf{Step II. }} Let $P':=\prod_{j:\mathcal{D}F_j\neq 0}F_j$, then $\mathcal{S}(P)\setminus Z(Q)\subset \mathcal{S}(P')$.
        We claim two facts.

        {\bf{Fact 1.}} $\gcd(P',\mathcal{D}P')=1$.

        For any $F_j\nmid \mathcal{D}F_j$, we can write 
        \[\mathcal{D}P'=\mathcal{D}\left(F_j\cdot\frac{P'}{F_j}\right)=(\mathcal{D}F_j)\frac{P'}{F_j}+F_j\mathcal{D}\left(\frac{P'}{F_j}\right).\]
        Then $\mathcal{D}P'\equiv (\mathcal{D}F_j)\frac{P'}{F_j}\mod F_j$.
        By $F_j\nmid \mathcal{D}F_j$, $\gcd(F_j,\frac{P'}{F_j})=1$ and the prime property of $F_j$ on $\mathfrak{K}[x,y,z]$, we have $F_j\nmid (\mathcal{D}F_j)\frac{P'}{F_j}$.
        Hence $F_j\nmid \mathcal{D}P'$.
        By the arbitrary of $F_j$ we obtain $\gcd(P',\mathcal{D}P')=1$.

        {\bf{Fact 2.}} $P'(x,y,a)\not\equiv 0$ for all $a\in\mathfrak{K}$.

        We write $P'(x,y,z)=\sum_{i,k}p_{ik}'(z)x^iy^k$.
        If $P'(x,y,a)\equiv 0$ for some $a\in\mathfrak{K}$, then $p'_{ik}(a)=0$ for all $i,k$.
        Then $(z-a)|p_{ik}'(z)\Rightarrow (z-a)|P'$.
        However, $\mathcal{D}(z-a)=0$, which contradicts to the definition of $P'$.

        \smallskip

        If $\deg_y P'=0$, then Fact 2 gives $P'(x,ax+b,a)\neq 0$.
        Thus $\mathcal{S}(P')=\varnothing$.
        At this time (\ref{thm:polynomialdec}) holds trivially.

        Now assume $u:=\deg_y(P')\geq 1$, $v:=\deg_{y}(\mathcal{D}P')\geq 0$, $d_1:=\deg_{x,y}P'\leq D_1$ and $d_2:=\deg_zP'\leq D_2$.
        Then $u\leq d_1$, $v\leq d_1-1$ and $\deg_z(\mathcal{D}P')\leq d_2+1$.
        Let $\mathfrak{K}(x)$ be the rational function field (with variable $x$), and 
        \[\mathfrak{K}(x)[z][y]_{<d}:=\left\{\sum_{j=0}^{d-1}a_j(z)y^j:a_j(z)\in\mathfrak{K}(x)[z]\right\}.\] 
        Consider $\mathfrak{K}(x)[z]$-linear map 
        \begin{align*} 
            \mathfrak{K}(x)[z][y]_{<v}\oplus \mathfrak{K}(x)[z][y]_{<u}&\to \mathfrak{K}(x)[z][y]_{<u+v}\\
            (A,B)&\mapsto AP'+B\mathcal{D}P'
        \end{align*}
        Its matrix $M(z)$ is called Sylvester matrix, which has size $(u+v)\times (u+v)$ with determinant $\det M(z)\not\equiv 0$.
        And $\deg_z\det M(z)\leq vd_2+u(d_2+1)\leq (d_1-1)d_2+d_1(d_2+1)\lesssim D_1D_2$.

        \medskip

        \noindent{\bf{Step III.}}
        Define $n_a:=|\{b\in\mathfrak{K}:P'(x,ax+b,a)\equiv 0\}|$, then $|\mathcal{S}(P')|=\sum_{a\in\mathfrak{K}}n_a$.
        We will prove that $(z-a)^{n_a}|\det M(z)$ in $\mathfrak{K}(x)[z]$.

        Fix $a\in\mathfrak{K}$.
        Since $P'(x,y,a)\not\equiv 0$ and $\deg_yP'=u$, each $b$ counted by $n_a$ gives a root $y=ax+b$.
        Thus $n_a\leq u$.
        Without loss of generality, we may assume $n_a$ is non-zero and enumerate the number of set $\{b_1,\dots,b_{n_a}\}$.
        Differentiating $P'(x,ax+b_i,a)\equiv 0$, we obtain 
        \[\partial_x P'(x,ax+b_i,a)+a\partial_yP'(x,ax+b_i,a)=\mathcal{D}P'(x,ax+b_i,a)\equiv 0.\]
        Hence $y=ax+b_i$ is the common root of $P'(x,y,a)$ and $\mathcal{D}P'(x,y,a)$.
        Define evaluation functionals 
        \begin{align*} 
            \omega_i:\mathfrak{K}(x)[y]_{<u+v}&\to\mathfrak{K}(x),\qquad\quad  1\leq i\leq n_a.\\
            h&\mapsto h(ax+b_i)
        \end{align*}
        For $M(a)(A,B)=A(y)P'(x,y,a)+B(y)\mathcal{D}P'(x,y,a)$, it's clear that $\omega_i(M(a)(A,B))=0$.
        Represent $\omega_i$ into row vector $(1, ax+b_i,\dots, (ax+b_i)^{u+v-1})$.
        The first $n_a$ coordinates of $\{\omega_i\}_{i=1}^{n_a}$ form a Vandermonde matrix, whose determinant equal to $\prod_{1\leq i<j\leq n_a}(b_j-b_i)\neq 0$,
        which means $\{\omega_i\}_{i=1}^{n_a}$ are linearly independent.
        Expend this family into a basis $\{e_j\}_{j=1}^{u+v}$ of the dual of $\mathfrak{K}(x)^{u+v}$.
        Let 
        \[T_a=\begin{pmatrix}
            \omega_1\\
            \vdots\\
            \omega_{n_a}\\
            e_{n_a+1}\\
            \vdots\\
            e_{u+v}
        \end{pmatrix}.\]
        For $1\leq i\leq n_a$, $(T_aM(a))_{ij}=\omega_i(M(a)e_j^T)=0$.
        Hence the first $n_a$ rows of $T_aM(a)$ are all zero:
        \[T_aM(a)=\begin{pmatrix}
            0&\cdots&0\\
            \vdots&&\vdots\\
            0&\cdots&0\\
            &*&\\
            &\vdots&\\
            &*&
        \end{pmatrix}.\]
        Applying $T_a$ to $M(z)$, we know that for $1\leq i\leq n_a$, $(T_aM(z))_{ij}$ has root $a$.
        Then 
        \[T_aM(z)=\begin{pmatrix}
            (z-a)I_{n_a}&0\\
            0&I_{u+v-n_a}
        \end{pmatrix}\widetilde{M_a}(z).\]
        Thus $(z-a)^{n_a}|\det(T_aM(z))=\det (T_a) \det M(z)$, which gives $(z-a)^{n_a}|\det M(z)$.

        \medskip

        In summary, $|\mathcal{S}(P')|=\sum_{a\in\mathfrak{K}}n_a\leq \deg_z\det M(z)\lesssim D_1D_2$.

    \end{proof}
    \section{Proof of Theorem \ref{thm:szemereditrotter}}\label{sec:incidence}
    In this section, we prove Theorem \ref{thm:szemereditrotter} and Theorem \ref{thm:rrichF}.
    We prove an estimate of $r$-rich points on algebraic variety first.
    \begin{lemma}\label{lem:algebraic}
        For algebraically closed field $\mathfrak{K}$ and $Q\in\mathfrak{K}[a,b]$ which is non-zero with degree $\leq D$.
        Assume $\mathcal{L}$ be collection of distinct lines in $\mathfrak{K}^2$.
        If finite set $S\subset Z(Q)$ such that each point in $S$ intersect with $r\geq 2$ lines in $\mathcal{L}$.
        Then $|S|\lesssim\frac{D|\mathcal{L}|}{r}$.
    \end{lemma}
    \begin{proof}
        Recall
        \[\mathcal{I}(S,\mathcal{L}):=\sum_{s\in S}|\{\ell\in\mathcal{L}:s\in\ell\}|=\sum_{\ell\in\mathcal{L}}|\{s\in S:s\in \ell\}|.\]
        It's clear that $\mathcal{I}(S,\mathcal{L})\geq r|S|$.
        On the other hand, by B\'ezout's theorem, if $\ell\nsubset Z(Q)$, then $|\ell\cap Z(Q)|\leq D$.
        Otherwise, there are at most $D$ lines contained in $Z(Q)$.
        And each lines contains at most $|\mathcal{L}|-1$ points in $S$.
        Thus $|\mathcal{I}(S,\mathcal{L})|\leq |\mathcal{L}|D+D(|\mathcal{L}|-1)\lesssim D|\mathcal{L}|$.
        Hence $|S|\lesssim\frac{D|\mathcal{L}|}{r}$.

    \end{proof}
    For $A\subset \mathbb{F}^2$, recall the definition of the collection of $r$-rich lines is
    \[\mathcal{L}_r(A):=\{\ell\subset \mathbb{F}^2:\ell\text{ is a line such that }|A\cap\ell|\geq r\}.\]

        
    \begin{proof}[Proof of Theorem \ref{thm:rrichF}]
        It suffices to consider $r\leq \min\{|A|,p\}$.

        \noindent {\bf{Step I.}} If $r\leq 2|A|^{1/2}$, then we will use Theorem \ref{thm:polynomialdec} to prove (\ref{eq:rrichF}).
        We construct polynomial $P$ in this step.

        Let $D_2=\lceil\frac{4|A|}{r^2}\rceil$, $D_1=D_2r-1<p$.
        We can find $P\in\mathbb{F}[x,y,z]$ with $\deg_{x,y}P\leq D_1$, $\deg_z P\leq D_2$ such that the coefficient of $t^j$ in $P(a_1+t,a_2+zt,z)$ equal to $0$ in $\mathbb{F}[z]$ for all $(a_1,a_2)\in A$ and $0\leq j<D_2$.

        This is because all polynomials $P\in\mathbb{F}[x,y,z]$ such that $\deg_{x,y}P\leq D_1$, $\deg_z P\leq D_2$ form a linear space with dimension $(D_2+1)\binom{D_1+2}{2}$.
        And the coefficient of $t^j$ in $P(a_1+t,a_2+zt,z)$ be zero as a polynomial in $\mathbb F[z]$ impose homogeneous linear equations:
        For a basis monomial, the coefficient of $t^j$ in $(a_1+t)^i(a_2+zt)^kz^l$ is
        \begin{equation}\label{eq:coefficient}
            \sum_{\substack{m+n=j\\0\leq m\leq i\\0\leq n\leq k}}\binom{i}{m}\binom{k}{n}a_1^{i-m}a_2^{k-n}z^{l+n}.
        \end{equation}
        Since $l\le D_2$ and $n\le j$, this is a polynomial in $z$ of degree at most $D_2+j$. (\ref{eq:coefficient}) equals to zero identically imposes at most $D_2+j+1$ homogeneous linear equations.
        So the total number of equations is at most 
        \[|A|\sum_{j=0}^{D_2-1}(D_2+j+1)=\frac{|A|D_2(3D_2+1)}{2}.\]
        The dimension of the polynomial space is strictly larger than this number.
        Hence there is a non-zero $P\in\mathbb{F}[x,y,z]$ satisfying all the imposed conditions and the required degree bounds.
        
        We consider non-vertical lines first.
        Assume $\ell: y=ax+b$ satisfies $|A\cap \ell|\geq r$.
        The polynomial $P(x,ax+b,a)$ has degree $\leq D_1$.
        For each $(a_1,aa_1+b)\in A\cap \ell$, since the coefficient of $t^j$ in $P(a_1+t,a(a_1+t)+b,a)$ is $0$ for $0\leq j<D_2$, $t^{D_2}|P(a_1+t,a(a_1+t)+b,a)$.
        Let $t=x-a_1$, we obtain $(x-a_1)^{D_2}|P(x,ax+b,a)$.
        And there are at least $r$ distinct $a_1$, which means $P(x,ax+b,a)$ is divisible by a polynomial of degree $D_2r>D_1$.
        Hence $P(x,ax+b,a)\equiv 0$, which implies $(a,b)\in \mathcal{S}(P)$.

        \medskip

        \noindent {\bf{Step II.}}

        Let $\mathfrak{K}=\bar{\mathbb{F}}=\cup_{m=1}^\infty\mathbb{F}_{p^m}$.
        Applying Theorem \ref{thm:polynomialdec} to $P$, there exists $Q\in\mathfrak{K}[a,b]$ such that $\deg Q\leq D_2$ and $|\mathcal{S}(P)\setminus Z(Q)|\lesssim D_1D_2$.
        In the dual space:
        for line set $\mathcal{L}_A:=\{\{(a,b)\in\mathfrak{K}^2:aa_1+b=a_2\}:(a_1,a_2)\in A\}$, by Lemma \ref{lem:algebraic}, the $r$-rich points on $Z(Q)$ $\lesssim \frac{D_2|A|}{r}$.
        Hence 
        \[|\{\ell\in\mathcal{L}_r(A):\ell {\text{ is non-vertical}}\}|\lesssim \frac{D_2|A|}{r}+D_1D_2\lesssim\frac{|A|^2}{r^3}.\]
        For vertical lines, since they are disjoint, there are at most $\frac{|A|}{r}$ of them to be $r$-rich.
        Hence $|\mathcal{L}_r(A)|\lesssim \frac{|A|^2}{r^3}+\frac{|A|}{r}$.

        \medskip

        \noindent{\bf{Step III.}} High rich part.

        If $r>2|A|^{1/2}$, we can prove (\ref{eq:rrichF}) by using C\'ordoba $L^2$ argument.
        For each $\ell\in\mathcal{L}_r(A)$, let $Y(\ell)$ be the union of $r$ points on $A\cap \ell$.
        Define $\mu_Y(x):=|\{\ell\in\mathcal{L}_r(A): x\in Y(\ell)\}|$.
        By double counting and the fact that distinct lines have at most one intersection point, 
        \[\sum_{x\in A}\mu_Y(x)=r|\mathcal{L}_r(A)|,\]
        \[\sum_{x\in A}\mu_Y(x)^2=\sum_{x\in A}\mu_Y(x)(\mu_Y(x)-1)+\sum_{x\in A}\mu_Y(x)\leq r|\mathcal{L}_r(A)|+|\mathcal{L}_r(A)|(|\mathcal{L}_r(A)|-1).\]
        By Cauchy-Schwarz, $r^2|\mathcal{L}_r(A)|^2\leq |A|(r|\mathcal{L}_r(A)|+|\mathcal{L}_r(A)|^2)$.
        Thus $|\mathcal{L}_r(A)|\lesssim\frac{|A|}{r}$.

    \end{proof}
    
    Now we can prove Theorem \ref{thm:szemereditrotter}.



    \begin{proof}[Proof of Theorem \ref{thm:szemereditrotter}]
        For $r\geq\max\{2,\frac{16|A|}{p}\}$, we claim that
        \begin{equation}\label{eq:partialincidence}
            \mathcal{I}(A,\mathcal{L})\lesssim r|\mathcal{L}|+\frac{|A|^2}{r^2}+|A|.
        \end{equation}
        In fact, $\sum_{\ell\in\mathcal{L}\setminus\mathcal{L}_r(A)}\mathcal{I}(A,\ell)\leq r|\mathcal{L}|$.
        If $r\leq |A|^{1/2}$, let 
        \[\mathcal{L}_j:=\{\ell\in\mathcal{L}:2^jr\leq |A\cap\ell|<\min\{2^{j+1}r,|A|^{1/2}\}\}.\]
        Then by Theorem \ref{thm:rrichF},
        \[\sum_{\ell\in\mathcal{L}_j}|A\cap \ell|\leq 2^{j+1}r|\mathcal{L}_{j}(A)|\lesssim \frac{|A|^2}{(2^jr)^2}+|A|\lesssim \frac{|A|^2}{(2^jr)^2}.\]
        Then the incidence in this part is 
        \[\sum_{j\geq 0}\mathcal{I}(A,\mathcal{L}_j)\lesssim \sum_{j\geq 0}\frac{|A|^2}{(2^jr)^2}\lesssim\frac{|A|^2}{r^2}.\]
        If $r\geq |A|^{1/2}$ and also for the remaining high rich part, by Theorem \ref{thm:rrichF} we have $|\mathcal{L}_r(A)|\lesssim \frac{|A|^2}{r^3}+\frac{|A|}{r}\lesssim |A|^{1/2}$.
        Let $\mu(x):=|\{\ell\in\mathcal{L}_r(A):x\in\ell\}|$, then 
        \[\mathcal{I}(A,\mathcal{L}_r(A))=\sum_{x\in A}\mu(x)\leq |A|+\sum_{x\in A:\mu(x)\geq 2}\binom{\mu(x)}{2}\leq |A|+\binom{|\mathcal{L}_r(A)|}{2}\lesssim |A|.\]
        In the third inequality, we use the facts that $\binom{\mu(x)}{2}$ counts the number of unordered pairs of lines in $\mathcal{L}_r(A)$ that pass through $x$ and distinct lines have at most one intersection point.
        So (\ref{eq:partialincidence}) holds.

        Now let $r=\max\{2,\frac{16|A|}{p},(|A|^2/|\mathcal{L}|)^{1/3}\}$, then 
        $r|\mathcal{L}|\lesssim |\mathcal{L}|+\frac{|A||\mathcal{L}|}{p}+|A|^{2/3}|\mathcal{L}|^{2/3}$ and $\frac{|A|^2}{r^2}\lesssim |A|^{2/3}|\mathcal{L}|^{2/3}$.
        Thus (\ref{eq:szemereditrotter}) holds.

    \end{proof}

    \begin{remark}
        We remark that by a similar process, we can also prove the Szemerédi-Trotter theorem for all characteristic zero field, e,g. $\mathbb{R}^2$, $\mathbb{C}^2$ and $\mathbb{Q}_p^2$.
    \end{remark}

    \section{Applications to Incidence Estimates}
    We prove Theorem \ref{thm:furstenberg} in this section first.
    \begin{proof}[Proof of Theorem \ref{thm:furstenberg}]
        If $r<\max\left\{2,\frac{16|E|}{p}\right\}$, then it's clear $|E|\lesssim pr$.
        Else, by Theorem \ref{thm:rrichF}, 
        \[|\mathcal{L}|\lesssim \frac{|E|^2}{r^3}+\frac{|E|}{r},\]
        which means 
        \[|E|\gtrsim r^{3/2}|\mathcal{L}|^{1/2}+r|\mathcal{L}|.\]
    \end{proof}

    And we prove a proposition that is useful in the restriction estimate in $\mathbb{F}^3$.

    \begin{proposition}\label{prop:rectangles}
        Assume $p\equiv 3\mod 4$.
        For $A\subset \mathbb{F}^2$, define
        \[R(A):=|\{(a,b,c,d)\in A^4: a+c=b+d, |a|^2+|c|^2=|b|^2+|d|^2\}|.\]
        Then 
        \begin{equation}\label{eq:rectangles}
            R(A)\lesssim |A|^2\log(|A|+1)+\frac{|A|^3}{p}.
        \end{equation}
    \end{proposition}
    \begin{proof}
        Since $p\equiv 3\mod 4$, the number of degenerate rectangles $2|A|^2-|A|\lesssim|A|^2$.
        Now we only consider the non-degenerate rectangles.
        Let $r:=\max\{2,\lceil\frac{16|A|}{p}\rceil\}$.

        Firstly, we count rectangles for which every edge lines fewer than $r$ points in $A$.
        The number of choices of adjacent vertices $a,b$ is $|A|^2$.
        The vertex $d$ is on the perpendicular line through $a$, which has fewer than $r$ points in $A$.
        Thus the number of these rectangles $\lesssim r|A|^2$. When $r=\lceil\frac{16|A|}{p}\rceil$, they contribute $\sim\frac{|A|^3}{p}$ rectangles.

        For the remaining rectangles, choose one of its edge lines with maximal intersection with $A$.
        Group these lines according to $2^jr\leq |A\cap \ell|<2^{j+1}r$ for $j=0,1,\cdots,\mathcal{O}(\log(|A|+1))$.
        By pigeonholing, it suffices to consider one class of $\ell$ with fixed $j$.
        Without loss of generality, we only consider the order that these $\ell$ pass through adjacent $a,b$.
        Then the number of choices of $a,b$ $\leq (2^{j+1}r)^2$.
        And the number of choices of $d$ $\leq 2^{j+1}r$.
        Hence each $\ell$ contributes $\mathcal{O}((2^{j}r)^3)$ rectangles.

        On the other hand, for each $a\in A\cap \ell$, let $\ell_a^\perp$ be the perpendicular line through $a$.
        Since $p\equiv 3\mod 4$, $\{\ell_a^\perp\}_{a\in A\cap\ell}$ are distinct parallel lines.
        Then $\sum_{a\in A\cap \ell}|A\cap\ell_a^\perp|\leq |A|$.
        For each $a,d$, there are $\leq 2^{j+1}r$ choices of $b\in\ell$. 
        Hence each $\ell$ contributes $\mathcal{O}(|A|2^{j}r)$ rectangles.

        Since $\ell\in\mathcal{L}_{2^{j}r}(A)$, if $2^j r\leq |A|^{1/2}$, using the first bound $\mathcal{O}((2^j r)^3)$ and Theorem \ref{thm:rrichF},
        \[R(A)\lesssim \log(|A|+1)(2^jr)^3|\mathcal{L}_{2^jr}(A)|\lesssim \log(|A|+1)(|A|^2+|A|2^{2j}r^2)\lesssim|A|^2\log(|A|+1).\]
        If $2^jr>|A|^{1/2}$, using the second bound $\mathcal{O}(|A|2^jr)$ and Theorem \ref{thm:rrichF},
        \[R(A)\lesssim \log(|A|+1)|A|2^{j}r|\mathcal{L}_{2^jr}(A)|\lesssim |A|^2\log(|A|+1).\]
    \end{proof}
    \begin{remark}
        We remark that if $p\equiv 1\mod 4$, then there exists $\mi\in\mathbb{F}$ such that $\mi^2=-1$.
        Select $A=\{(a,a\mi):a\in\mathbb{F}\}$, then $|(a,a\mi)|^2=a^2+(a\mi)^2=0$.
        For any $a,b,c\in A$, let $d=a+c-b$.
        The both equations in $R(A)$ are satisfied, thus $R(A)=p^3$.
        So (\ref{eq:rectangles}) fails.

        And when $p\equiv 3\mod 4$, (\ref{eq:rectangles}) is sharp up to constant.
        Let $A=\mathbb{F}^2$, then $R(A)\sim p^5$.
        On the other hand, for any $4\leq N\leq \sqrt{p}$, let $A=\{0,1,\dots,N-1\}^2$.
        For each $q<N-1$, we can write $q=r+s$ for $\varphi(q)$ pairs coprime integers $r,s\geq 1$.
        Here $\varphi(q)$ is Euler totient function.
        Then 
        \[x,\;x+k(r,s),\; x+l(-s,r),\; x+k(r,s)+l(-s,r)\]
        will form a rectangle in $A$ for any $x\in A$ and integers $1\leq k,l\lesssim \frac{N}{q}$.
        Hence 
        \[R(A)\gtrsim N^2\sum_{q<N-1}\varphi(q)\frac{N^2}{q^2}= N^4\sum_{q<N-1}\frac{\varphi(q)}{q^2}\sim N^4\log N.\]
        Here we use a classical approximation \cite[Chapter 3, Exercise 6]{zbMATH03523640}:
        \[\sum_{q<N}\frac{\varphi(q)}{q^2}=\frac{\pi^2}{6} \log N+\mathcal{O}(1).\]
        
    \end{remark}

    \section{Applications to Sum-Product estimates}
    In this section, we prove Theorem \ref{thm:sumproduct} and Theorem \ref{thm:diffproduct}.

    \medskip

    \noindent{\bf{Theorem \ref{thm:sumproduct}.}}
        For $A\subset \mathbb{F}$, we have 
        \[\max\{|A+A|,|A\cdot A|\}\gtrsim\min\{(p|A|)^{1/2},|A|^{5/4}\}.\]
        In particluar, if $|A|\leq p^{2/3}$, then 
        \[\max\{|A+A|,|A\cdot A|\}\gtrsim|A|^{5/4}.\]

    \medskip
    \begin{proof}
        Consider $P=(A+A)\times (A\cdot A)$ and $|A|^2$ lines $\mathcal{L}=\{\ell_{a,b}:y=a(x-b):a,b\in A\}.$
        For every $a,b,c\in A$, $(b+c,ac)\in P\cap \ell_{a,b}$.
        Hence $\mathcal{I}(P,\mathcal{L})\geq |A|^3$.

        On the other hand, Theorem \ref{thm:szemereditrotter} impies 
        \[|A|^3\leq\mathcal{I}(P,\mathcal{L})\lesssim \frac{|P||A|^2}{p}+|P|^{2/3}|A|^{4/3}+|P|+|A|^2.\]
        Hence $|P|=|A+A||A\cdot A|\gtrsim\min\{p|A|,|A|^{5/2}\}$.
        In particluar, 
        \[\max\{|A+A|,|A\cdot A|\}\gtrsim\min\{(p|A|)^{1/2},|A|^{5/4}\}.\]
    \end{proof}

    \noindent{\bf{Theorem \ref{thm:diffproduct}}}
        For $A\subset \mathbb{F}$ satisfies $|A|\geq 2$, there exists a subset $B\subset A^2$ with $|B|\geq |A|^2/2$ such that for any $(a,b)\in B$, 
        \[|(A-a)\cdot(A-b)|\gtrsim \min\left\{\frac{|A|^2}{\log|A|},p\right\}.\]
        In particlur, 
        \[|(A-A)\cdot(A-A)|\gtrsim \min\left\{\frac{|A|^2}{\log|A|},p\right\}.\]
    
    \medskip

    \begin{proof}
        {\bf{Step I.}}
        For $A\subset \mathbb{F}$, let $P=A\times A$.
        Define 
        \[T(A):=|\{(z_1,z_2,z_3)\in P^3: {\text{ there exists line }}\ell \text{ such that }z_1,z_2,z_3\in\ell\}|.\]
        We will prove 
        \[T(A)\lesssim |A|^4\log|A|+\frac{|A|^6}{p}.\]

        For each line $\ell$ in $\mathbb{F}^2$, let $Y(\ell)=P\cap\ell$.
        By the fact that two distinct points determine only one line, we have $\sum_{\ell}|Y(\ell)|(|Y(\ell)-1|)=|P|(|P|-1)$.
        Hence 
        \[T(A)\lesssim |P|^2+\sum_{\ell}\binom{|Y(\ell)|}{3}.\]
        Let $r=\max\{2,\lceil \frac{16|P|}{p}\rceil\}$, then 
        \[\sum_{\ell:|Y(\ell)|\leq r}\binom{|Y(\ell)|}{3}\lesssim r|P|^2\lesssim|P|^2+\frac{|P|^3}{p}.\]
        For the remaining lines, group them according to $2^jr<|Y(\ell)|\leq 2^{j+1}r$, $ j=0,1,\dots,\mathcal{O}(\log|A|)$.
        For each $j$, by Theorem \ref{thm:rrichF}, 
        \[|\mathcal{L}_{2^jr}(P)|\lesssim \frac{|P|^2}{(2^jr)^3}+ \frac{|P|}{2^jr}.\]
        Then 
        \[\sum_{\ell:|Y(\ell)|>r}\binom{|Y(\ell)|}{3}\lesssim \sum_{j=0}^{\mathcal{O}(\log|A|)}(2^jr)^3|\mathcal{L}_{2^jr}|\lesssim |P|^2\log|A|.\]
        In summary, 
        \[T(A)\lesssim |P|^2+\sum_{\ell}\binom{|Y(\ell)|}{3}\lesssim |P|^2\log|A|+\frac{|P|^3}{p}=|A|^4\log|A|+\frac{|A|^6}{p}.\]

        \medskip

        {\bf{Step II.}}
        For any $U,V\subset \mathbb{F}$, define the multiplicative energy 
        \[E(U,V):=|\{(u_1,u_2,v_1,v_2)\in U^2\times V^2:u_1v_1=u_2v_2\}|.\]
        Then 
        \begin{equation}\label{eq:collinear}
            \sum_{a,b\in A}E(A-a,A-b)\leq T(A).
        \end{equation}
        In fact, the L.H.S. of (\ref{eq:collinear}) counts the number of pairs $(a,b,x_1,x_2,y_1,y_2)\in A^6$ such that 
        \[(x_1-a)(y_1-b)=(x_2-a)(y_2-b).\]
        In other words, $(a,b), (x_1,y_2), (x_2,y_1)$ are on the same non-vertical line.
        Thus 
        \[\sum_{a,b\in A}E(A-a,A-b)\leq T(A)\lesssim |A|^4\log|A|+\frac{|A|^6}{p}.\]

        Define the average
        \[M:=\frac{1}{|A|^2}\sum_{a,b\in A}E(A-a,A-b)\lesssim |A|^2\log|A|+\frac{|A|^4}{p}.\]
        Then we can choose  
        \[B:=\{(a,b)\in A^2: E(A-a,A-b)\leq 2M\}\]
        such that $|B|\geq |A|^2/2$.
        For each $(a,b)\in B$, by Cauchy-Schwarz 
        \[|A|^4\leq |(A-a)\cdot(A-b)|\,E(A-a,E-b)\lesssim |(A-a)\cdot(A-b)|\left(|A|^2\log|A|+\frac{|A|^4}{p}\right). \]
        So
        \[|(A-a)\cdot(A-b)|\gtrsim \min\left\{\frac{|A|^2}{\log|A|},p\right\}.\]
        Note that $(A-a)\cdot(A-b)\subset (A-A)\cdot(A-A)$, we obtain 
        \[|(A-A)\cdot(A-A)|\gtrsim \min\left\{\frac{|A|^2}{\log|A|},p\right\}.\]
    \end{proof}

    \section{Restriction Estimate in \texorpdfstring{$\mathbb{F}^3$}{F3}}\label{sec:fourier}
    In this section, we prove Theorem \ref{thm:restriction}.
    Firstly, we introduce some basic notation about the fourier analysis in finite field, which could be found in \cite{zbMATH02103577}.

    For function $f:\mathbb{F}^n\to \mathbb{C}$, define the integral under counting measure 
    \[\int_{\mathbb{F}^n}f(x)\opd x=\sum_{x\in\mathbb{F}^n}f(x).\]
    Let $e:\mathbb{F}^n\to\mathbb{C}^\times$ be a non-trivial character with respect to $(\mathbb{F}^n,+)$.
    Then all of these additive characters are unitary.
    Let $\hat{\mathbb{F}}^n$ be the collection of all these characters, equipped with weak topology.
    For a fixed character $e$, we choose $e(a):=\me^{2\pi \mi a/p}$.
    All elements in $\hat{\mathbb{F}}^n$ can be given by the isomorphism:
    \begin{align*} 
        \mathbb{F}^n&\to \hat{\mathbb{F}}^n\\ \xi&\mapsto e_{\xi}:=e(\left<\cdot,\xi\right>)
    \end{align*}
    Hence we can write the Fourier transform of function $f$
    \[\hat{f}(\xi)=\int_{\mathbb{F}^n}f(x)e(-x\cdot\xi)\]
    to be function on the dual space $\mathbb{F}_*^n$.
    For function $g:\mathbb{F}_*^n\to\mathbb{C}$, define the integral under normalized counting measure:
    \[\int_{\mathbb{F}_*^n}g(\xi)\opd \xi=p^{-n}\sum_{\xi\in\mathbb{F}_*^n}g(\xi).\]
    
    Consider paraboloid $\mathbb{P}^{2}:=\{(\xi,|\xi|^2):\xi\in\mathbb{F}_*^{2}\}$ in $\mathbb{F}_*^3$, where $|\xi|^2:=\xi\cdot\xi$.
    To make the paraboloid anisotropic, we assume $p\equiv 3 \mod 4$.
    Thus $-1$ is not square.
    Define the Fourier extension operator 
    \[(f\opd\sigma)^\vee(x):=\frac{1}{p^{2}}\sum_{\xi\in\mathbb{P}^{2}}f(\xi)e(x\cdot\xi).\]
    The inverse Fourier transform of the surface measure can be written as:
    \[
    (\opd \sigma)^\vee(x)=
    \begin{cases}
        \frac{1}{p^{2}}e(-\frac{|\bar{x}|^2}{4x_3})\left(\sum_{\xi\in\mathbb{F}_*}e(x_3\xi^2)\right)^{2},\quad x_3\neq 0\\
        \delta_0(\bar{x}),\quad x_3=0
    \end{cases}.\]
    Define the Bochner-Riesz kernel $K(x):=(\opd\sigma)^\vee(x)-\delta_0(x)$.
    Thus $|K(x)|=p^{-1},\;x_3\neq 0$ and $K(x)=0,\; x_3=0$.
    By Plancherel,
    \begin{equation}\label{eq:plancherel}
        \|\hat{f}\|_{L^2(\mathbb{P}^2,\opd\sigma)}^2=\left<f,f*(\opd\sigma)^\vee\right>=\|f\|_{L^2}^2+\left<f,f*K\right>.
    \end{equation}
    Mockenhaupt-Tao have proved the following proposition:
    \begin{proposition}
        Assume $G\subset \mathbb{F}^3$ and $|f|\leq \chi_G$.
        For each $z\in\mathbb{F}$, define the slice 
        \[G_z:=\{\bar{x}\in\mathbb{F}^2:(\bar{x},z)\in G\}\] 
        and 
        \[f_z(x):=\begin{cases} f(\bar{x},z),\quad x_3=z\\ 0,x_3\neq z\end{cases}.\]
        Then 
        \[\|f_z*K\|_{L^4}^4\leq p^{-1}R(G_z).\]
        Here 
        \[R(A):=|\{(a,b,c,d)\in A^4: a+c=b+d, |a|^2+|c|^2=|b|^2+|d|^2\}|\] is the number of ordered rectangles in $A$.
    \end{proposition}

    Now we prove Theorem \ref{thm:restriction}.

    \medskip

    \noindent{\bf{Theorem \ref{thm:restriction}.}}
    If $p\equiv 3\mod 4$, then $R^*(2\to\alpha)$ holds for $\alpha>\frac{10}{3}$.

    \smallskip

    \begin{proof}
        We prove the dual of $R^*(2\to \alpha)$ for $\alpha>\frac{10}{3}$.
        For $1< s<\frac{10}{7}$, without loss of generality, we can assume $\|f\|_{L^s}=1$.
        Let
        \[E_j:=\{x:2^{-j-1}<|f(x)|\leq 2^{-j}\},\; f_j:=2^jf\chi_{E_j},\quad j\geq 0.\]
        Then $|E_j|\leq 2^{(j+1)s}$ and $|f_j|\leq \chi_{E_j}$.
        For each $f_j$, by (\ref{eq:plancherel}) we have
        \begin{align*} 
            \|\hat{f}_j\|_{L^2(\mathbb{P}^2,\opd\sigma)}^2&\leq |E_j|+\|f_j\|_{L^{4/3}}\|f_j*K\|_{L^4}\\
            &\leq |E_j|+|E_j|^{3/4}\sum_{z\in\mathbb{F}}\|f_{j,z}*K\|_{L^4}\\
            &\leq |E_j|+|E_j|^{3/4}p^{-1/4}\sum_{z\in\mathbb{F}}R(E_{j,z})^{1/4}.
        \end{align*}
        By Proposition \ref{prop:rectangles} and H\"older,
        \begin{align*} 
            \sum_{z\in\mathbb{F}}R(E_{j,z})^{1/4}&\lesssim\sum_{z\in\mathbb{F}}\left(|E_{j,z}|^{1/2}(\log(|E_{j}|+1))^{1/4}+\frac{|E_{j,z}|^{3/4}}{p^{1/4}}\right)\\
            &\leq (\log(|E_j|+1))^{1/4}p^{1/2}|E_j|^{1/2}+|E_j|^{3/4}.
        \end{align*}
        So 
        \begin{equation}\label{eq:bound1}
            \|\hat{f}_j\|_{L^2(\mathbb{P}^2,\opd \sigma)}\lesssim |E_j|^{1/2}+\log(|E_{j}|+1)^{1/8}p^{1/8}|E_j|^{5/8}+p^{-1/8}|E_j|^{3/4}.
        \end{equation}
        On the other hand, by (\ref{eq:plancherel}) we have 
        \begin{equation}\label{eq:bound2}
            \|\hat{f}_j\|_{L^2(\mathbb{P}^2,\opd \sigma)}\leq |E_j|^{1/2}+p^{-1/2}|E_j|.
        \end{equation}
        And by Parseval identity, 
        \begin{equation}\label{eq:bound3}
            \|\hat{f}_j\|_{L^2(\mathbb{P}^2,\opd \sigma)}^2=p^{-2}\sum_{\xi\in\mathbb{P}^2}|\hat{f}_j(\xi)|^2\leq p\|f_j\|_{L^2}^2\leq p|E_j|.
        \end{equation}
        Combining (\ref{eq:bound1}), (\ref{eq:bound2}) and (\ref{eq:bound3}), we obtain 
        \[\|\hat{f}_j\|_{L^2(\mathbb{P}^2,\opd \sigma)}\lesssim \log(|E_{j}|+1)^{1/8}|E_j|^{7/10}.\]
        Summing all $j$ up, we obtain
        \[\|\hat{f}\|_{L^2(\mathbb{P}^2,\opd\sigma)}\lesssim \sum_{j\geq 0}2^{-j}\log(|E_j|+1)^{1/8}|E_j|^{7/10}\lesssim \sum_{j\geq 0}j^{1/8}2^{-j(1-7s/10)}<\infty\]
        holds for $s<\frac{10}{7}$.
        By duality, $R^*(2\to \alpha)$ holds for $\alpha>\frac{10}{3}$.

    \end{proof}

    \section*{Acknowledgement}
    This project was supported by the National Key R\&D program of China: No.2022YFA1005700. C. Miao was supported by NSFC Grant 12371095 and 12531005.
    The second author have presented this paper at the Advanced Seminar on Harmonic Analysis at Peking University.
    The authors would like to thank the seminar participants, especially Yudong Liang and Xuanjun Luo, for comments to a previous draft of this manuscript.

    After completing this paper, we learned that Prof. Lewko had independently obtained similar results \cite{lewko2026szemereditrottertheoremarbitraryfields}.
    We are grateful to Prof. Lewko for sharing his manuscript and valuable suggestions.

    \bibliographystyle{alpha}
    \bibliography{bibli}
\end{document}